\documentclass[10pt]{article}
\usepackage{amsfonts, amsmath, amssymb, amscd, latexsym}
\usepackage{geometry}                
\usepackage{graphicx}
\usepackage{amssymb}
\usepackage{epstopdf}
\usepackage{hyperref}
\usepackage[utf8]{inputenc}
\usepackage{tikz}
\usepackage{mathtools}
\usepackage{array}

\usepackage{calrsfs}
\usepackage[cal=boondoxo]{mathalfa}

\usepackage{amsthm}
\usepackage{multirow}
\usepackage{tabu}
\usepackage{diagbox}

\usepackage{comment}
\usepackage{breqn}
\usepackage{nth}

\newtheorem{thm}{Theorem}
\newtheorem{lem}[thm]{Lemma}
\newtheorem{cor}[thm]{Corollary}
\newtheorem{prop}[thm]{Proposition}

\theoremstyle{definition}
\newtheorem{definition}{Definition}

\theoremstyle{definition}

\newcommand{\Z}{\mathbb{Z}}
\newcommand{\R}{\mathbb{R}}
\newcommand{\N}{\mathbb{N}}

\DeclarePairedDelimiter\ceil{\lceil}{\rceil}
\DeclarePairedDelimiter\floor{\lfloor}{\rfloor}

\newcommand{\cfr}[2]{\ceil*{\frac{#1}{#2}} }
\newcommand{\lfr}[2]{\floor*{\frac{#1}{#2}} }
\newcommand{\fra}[2]{\{#1; #2\}}

\author{Francisco Franco Munoz}
\date{}							% Activate to display a given date or no date
\title{\bf On the counting function for numerical monoids}

\begin{document}

\newcommand{\Addresses}{{% additional braces for segregating \footnotesize
  \bigskip
  \footnotesize

Francisco Franco Munoz, \textsc{Department of Mathematics, University of Washington,
    Seattle, WA 98195}\par\nopagebreak
  \textit{E-mail address}: \texttt{ffm1@uw.edu}
}}

\maketitle

\abstract{We provide explicit expressions for the counting function and main terms for numerical monoids.}

\section{Introduction}

The study of numerical monoids is a huge area of research (see \cite{RS} and references therein). A numerical monoid $M$ is a subset of $\N$ containing $0$ that is closed under addition and finite complement. In other words, a finite complement submonoid of $(\N,+)$. 

\noindent The aim of this paper is to calculate in an elementary way the main term of the counting function of a monoid (Theorems \ref{thm-5}, \ref{thm-13}, \ref{thm-14}) which is a natural object. We'll avoid the use of more advanced treatments such as quasi-polynomial functions and integer points of polytopes (\cite{BLDPS}) that provide a more comprehensive view of the subject. The paper is essentially self-contained and requires no prior knowledge of the subject. 

\subsection{Notation}

$\N = \{0, 1, 2, 3, ...\}$, the natural numbers \\
$\Z= -\N \cup \N$, the integers\\
$\N_{+} = \{1, 2, 3, ...\}$, the positive integers 

\section{Setting}

\noindent Let $M = \langle a_1, \dots, a_{n-1}, a_n \rangle$ be a submonoid of $(\N, +)$. By definition $M$ is called a numerical monoid if its complement is finite. Denote by $M_{\widehat{a_n}} = \langle a_1, \dots, a_{n-1} \rangle $ the submonoid obtained by dropping the generator $a_n$.

\begin{definition} For $x\in \R$, $a\in \N_{+}$, define $\displaystyle \fra{x}{a} = x-a\lfr{x}{a} $. 
\end{definition}

With this definition, we have $\{x\} = \fra{x}{1}$ the fractional part of $x$. 

\begin{prop}\label{prop-1}
Let $x, k\in \Z$, $a, b\in \N_{>0}$

\begin{enumerate}

\item $\fra{x}{a}$ is the residue of $x$ modulo $a$

\item $0\leq \fra{x}{a} < a$

\item $\fra{x}{a}=0 \iff a | x$. 

\item $\fra{x+ka}{a}=\fra{x}{a}$

\item $\fra{x}{b} = \fra{\fra{x}{ab}}{b}$

\end{enumerate}
\end{prop}

\begin{proof} Immediate.
\end{proof}

\begin{lem} \label{lem-2} 
For $y\in \N$, if $y<a_n$, then $y\in M \iff y\in M_{\widehat{a_n}}$. In particular, for $x\in \N$,  $\fra{x}{a_n} \in M \iff \fra{x}{a_n} \in M_{\widehat{a_n}}$.
\end{lem}

\begin{proof} This is clear.
\end{proof}

\begin{definition} We define the counting function $$T_{M}(x) = \# \{(x_1, ... , x_n) \in \N^n  \mid  x_1a_1 + ... + x_na_n = x \}$$
\end{definition}

The Theory of Erhart polynomials \cite{BLDPS} provides ways to determine $T_M$. Instead we'll address them directly in an elementary way.

\begin{thm} 
 $$ T_{M}(x)=T_{M}(x-a_n)+T_{M_{\widehat{a_n}}}(x)$$  
\end{thm}

\begin{proof} This is immediate since the set $ \{(x_1, ... , x_n) \in \N^n  \mid  x_1a_1 + ... + x_na_n = x \}$ can be partitioned as  $ \{(x_1, ... , x_n) \in \N^n  \mid  x_1a_1 + ... + x_na_n = x \} = $  $$=\{(x_1, ... , x_n) \in \N^n  \mid  x_n > 0 , x_1a_1 + ... + x_na_n = x \}\sqcup  \{(x_1, ... , x_n) \in \N^n  \mid  x_n=0, x_1a_1 + ... + x_na_n = x\}$$ $= \{(x_1, ... , x_{n-1}, y_n) \in \N^n  \mid  y_n=x_n-1\geq 0 , x_1a_1 + ... + x_{n-1}a_{n-1}+y_na_n = x-a_n \}\sqcup \{(x_1, ... , x_{n-1}, 0) \in \N^n  \mid  x_n=0, x_1a_1 + ... + x_{n-1}a_{n-1}+0 = x\}$ and the counting for each is $T_{M}(x-a_n)+T_{M_{\widehat{a_n}}}(x)$.

\end{proof}

%cor 4 
\begin{cor}\label{cor-4}

 $\displaystyle T_{M} (x) = \sum_{k=0}^{\lfr{x}{a_n}} T_{M_{\widehat{a_n}}}(x-ka_n) =  \sum_{l=0}^{\lfr{x}{a_n}} T_{M_{\widehat{a_n}}}(\fra{x}{a_n}+la_n) $
 \end{cor}
 
\begin{proof}

$\displaystyle T_{M} (x) - T_{M}\left(x-a_n\floor*{\frac{x}{a_n}}\right) =  \sum_{k=0}^{ \lfr{x}{a_n} -1} \left[T_{M}(x-ka_n)-T_{M}(x-(k+1)a_n)\right] \\=\sum_{k=0}^{\lfr{x}{a_n} - 1} T_{M_{\widehat{a_n}}}(x-ka_n)$. Also $T_{M}\left(x-a_n\floor*{\frac{x}{a_n}}\right)=T_{M}(\fra{x}{a_n})=T_{M_{\widehat{a_n}}}(\fra{x}{a_n}) = T_{M_{\widehat{a_n}}}\left(x-a_n\floor*{\frac{x}{a_n}}\right)$ by Lemma \ref{lem-2} and we conclude that $\displaystyle T_{M} (x) = \sum_{k=0}^{\lfr{x}{a_n}} T_{M_{\widehat{a_n}}}(x-ka_n)$ and making the change of variable $k=\lfr{x}{a_n}-l$ we get the last equality.

\end{proof}

%%%%%%%%%%%%%%%%%%%%%%%%%%%%%%%%%%%%%%%%%%%%%%%%%%%%%%%%%%%%%%%%%%%%%%%%%%%%%%%%%%%%%%%%%%%%

\section{Dimension $n=2$}

$M = \langle a, b \rangle$ and assume that $\gcd(a,b)=1$ in this subsection.

\begin{thm} \label{thm-5}
We have $$ T_{\langle a, b \rangle} (x) =  \lfr{x}{ab} + \epsilon_{\langle a, b\rangle}(\fra{x}{ab})$$ where for $0\leq y < ab$, $\epsilon(y)=\delta(y\in \langle a, b\rangle)$ is the indicator function. 
\end{thm}

\begin{proof}
Applying Corollary \ref{cor-4} with $n=2$ and $a_2=b$, we have $\displaystyle T_{\langle a, b \rangle}(x) =  \sum_{l=0}^{\lfr{x}{b}} T_{M_{\widehat{b}}}(\fra{x}{b}+lb) = \sum_{l=0}^{\lfr{x}{b}} T_{\langle a\rangle}(\fra{x}{b}+lb)$. Now, it's clear that $T_{\langle a \rangle }(y)$ is $1$ or $0$ according to whether $a|y$ or not and so $T_{\langle a\rangle}(y+ka)=T_{\langle a\rangle}(y)$. Let $y=\lfr{x}{b}$. By Proposition \ref{prop-1}, $\lfr{\lfr{x}{b}}{a}=\lfr{x}{ab}$, so $y=\fra{y}{a} + a\lfr{x}{ab}$. We split the sum $\displaystyle \sum_{l=0}^{\lfr{x}{b}} T_{\langle a\rangle}(\fra{x}{b}+lb) =  \sum_{i=0}^{\lfr{x}{ab} - 1} \sum_{j=0}^{a-1} T_{\langle a\rangle}\left(\fra{x}{b}+(ai+j)b\right) + \sum_{j=0}^{\fra{y}{a}} T_{\langle a\rangle}(\fra{x}{b}+(a\lfr{x}{ab}+j)b) =   \sum_{i=0}^{\lfr{x}{ab} - 1} \sum_{j=0}^{a-1} T_{\langle a\rangle}(\fra{x}{b}+jb) + \sum_{j=0}^{\fra{y}{a}} T_{\langle a\rangle}(\fra{x}{b}+jb) = \lfr{x}{ab}\sum_{j=0}^{a-1} T_{\langle a\rangle}(\fra{x}{b}+jb) + \sum_{j=0}^{\fra{y}{a}} T_{\langle a\rangle}(\fra{x}{b}+jb)$ and the proof follows from the next two lemmas.
\end{proof}

\begin{lem} For any $y\in \N$,  $\displaystyle \sum_{j=0}^{a-1} T_{\langle a\rangle}(y+jb)=1$
\begin{proof} Since $\gcd(a,b)=1$, there's unique  $j, 0\leq j\leq a-1$ such that $a\mid (y+jb)$.
\end{proof}
\end{lem}

\begin{lem} For any $x\in \N$, then $\displaystyle \sum_{j=0}^{\fra{y}{a}} T_{\langle a\rangle}(\fra{x}{b}+jb) = \epsilon_{\langle a, b\rangle}(\fra{x}{ab})$ where $y=\lfr{x}{b}$.
\end{lem}

\begin{proof} Omitted. 
\end{proof}

\subsection{A different proof}

\begin{lem}\label{lem-8} If $x=\alpha a+\beta b\in M=\langle a, b\rangle$, then $T_{M}(x)=\lfr{\alpha}{b}+\lfr{\beta}{a}+1$
\end{lem}

\begin{proof} Fix a representation  $x=\alpha a+\beta b$. We need to count representations $x=\alpha_1 a+\beta_1 b$. But $\alpha a+\beta b=\alpha_1 a+\beta_1 b \iff (\alpha-\alpha_1)a=(\beta_1-\beta)b$ and since $\gcd(a,b)=1$ this happens $\iff a\mid (\beta_1-\beta), b\mid (\alpha-\alpha_1)$ and $\beta_1-\beta=ja, \alpha-\alpha_1 = jb$. So we need to count the possibilities for $j\in \Z$. We have $\beta_1=ja+\beta$, $\alpha_1=-jb+\alpha$. And the condition is exactly $\beta_1\geq 0$ and $\alpha_1\geq 0$, which happens $\iff ja+\beta\geq 0, -jb+\alpha\geq 0 \iff \frac{-\beta}{a}\leq j\leq \frac{\alpha}{b} \iff \cfr{-\beta}{a}\leq j\leq \lfr{\alpha}{b} $. Now we have $ \cfr{-\beta}{a} = - \lfr{\beta}{a}$ hence there are $\lfr{\alpha}{b}+\lfr{\beta}{a}+1$ such $j$.
\end{proof}

\begin{cor}\label{cor-9} For $x<ab$, $T_{M}(x)=1 \iff x\in M$.
\end{cor}

\begin{proof} It's enough to show that if $x\in M$ and $x<ab$, then $T_{M}(x)=1$. By Lemma \ref{lem-8}, this holds iff $\lfr{\alpha}{b}=\lfr{\beta}{a}=0$ where $x=\alpha a+\beta b$ is some representation. But since $x< ab$, then $\alpha a\leq x < ab$ so $\alpha<b$ and $\beta b\leq x < ab$ so $\beta<a$  so 
$\lfr{\alpha}{b}=\lfr{\beta}{a}=0$.
\end{proof}

\begin{lem} \label{lem-10} If $x\in M$, then for all $k\in \N$, $T_{M}(x+kab)=T_{M}(x)+k$
\end{lem}

\begin{proof} By Lemma \ref{lem-8}, take  $x=\alpha a+\beta b\in M$ then $x+kab\in M$ and $x+kab=(\alpha+kb) a+\beta b$ and so $T_M(x+kab)=\lfr{\alpha+kb}{b}+\lfr{\beta}{a}+1=k+\lfr{\alpha}{b}+\lfr{\beta}{a}+1= T_{M}(x)+k$.
\end{proof}

\begin{lem}\label{lem-11} If $x\notin M$, $0\leq x\leq ab-1$, then $T_{M}(x+ab)=1$
\end{lem}

\begin{proof} Since $x+ab\geq ab$ and the Frobenius number of $M$ is $ab-a-b$ (\cite{RS}), $x+ab$ belongs to $M$, so $T_{M}(x+ab)\geq 1$. Now, as in Corollary \ref{cor-9}, it's enough to show that in any representation of $x+ab=\alpha a + \beta b$, both $\alpha < b$ and $\beta < a$. If not, say $\alpha\geq b$, then $x=\alpha a + \beta b -ab = (\alpha-b)a + \beta b$ which says that $x\in M$, contradiction. In the same way, $\beta < a$ and we're done.
\end{proof}

%\begin{lem} If $x\in \N$, $0\leq x\leq ab-1$, then the formula is valid.
%\end{lem}

\begin{thm} For all $x\in \N, T_{M}(x+ab)=T_{M}(x)+1$ and Theorem \ref{thm-5} is valid. 
\end{thm}
\begin{proof} By contradiction, take $x$ minimal such that this formula is not valid for $x$. Notice that by Lemma \ref{lem-11}, we have $x\geq ab$, and by Lemma \ref{lem-10}, we have $x\notin M$. Let $y=x-ab < x$ then by minimality, the theorem is valid for $y$ and so $T_{M}(x)=T_{M}(y+ab)=T_{M}(y)+1\geq 1$ in particular $x\in M$, contradiction. \\
Since $\forall x\in \N, T_{M}(x+ab)=T_{M}(x)+1$, by induction $\forall x, k\in \N$, $T_{M}(x+kab)=T_{M}(x)+k$, and so $T_{M}(x)=T_{M}(\fra{x}{ab}+ab\lfr{x}{ab})= \lfr{x}{ab} + T_{M}(\fra{x}{ab})= \lfr{x}{ab} + \epsilon_{\langle a, b\rangle}(\fra{x}{ab})$ the last equality since $\fra{x}{ab} < ab$ and using Corollary \ref{cor-9}. This proves the theorem. 

\end{proof}

%%%%%%%%%%%%%%%%%%%%%%%%%%%%%%%%%%%%%%%%%%%%%%%%%%%%%%%%%%%%%%%%%%%%%%%%%%%%%%%%%%%%%%%%%%%%

\section{Dimension $n=3$}

Assume $\gcd(a,b)=\gcd(a,c)=\gcd(b,c)=1$. These conditions are stronger than necessary, but the theory of numerical monoids permits one to reduce to this case in any dimension \cite{RS}. 
\vspace{2 mm}

Let $M = \langle a, b, c \rangle$, $M_{\widehat{c}}= \langle a, b \rangle$ dimension $3$ and $2$ numerical monoids.\\ 

Recall $\displaystyle T_{M} (x) = \sum_{l=0}^{\lfr{x}{c}} T_{\langle a, b\rangle}(\fra{x}{c}+lc)$

Let $y=\lfr{x}{c}$. Then $y= \fra{y}{ab} + ab\lfr{x}{abc}$. By Theorem \ref{thm-5}, $T_{\langle a, b\rangle}$ splits as sum, consequently we split the sum $\displaystyle  \sum_{l=0}^{\lfr{x}{c}} T_{\langle a, b\rangle}(\fra{x}{c}+lc) = \sum_{l=0}^{\lfr{x}{c}} \lfr{\fra{x}{c}+lc}{ab} + \sum_{l=0}^{\lfr{x}{c}}  \epsilon_{\langle a, b\rangle}(\fra{\fra{x}{c}+cl}{ab})= S1 + S2$\\

$S1= \displaystyle \sum_{l=0}^{\lfr{x}{c}} \lfr{\fra{x}{c}+lc}{ab} =   \sum_{i=0}^{\lfr{x}{abc} - 1} \sum_{j=0}^{ab-1} \lfr{\fra{x}{c}+(abi+j)c}{ab} +  \sum_{j=0}^{\fra{y}{ab}} \lfr{\fra{x}{c}+(ab\lfr{x}{abc}+j)c}{ab} \\= \sum_{i=0}^{\lfr{x}{abc} - 1} \sum_{j=0}^{ab-1} \lfr{\fra{x}{c}+jc}{ab}  +  \sum_{i=0}^{\lfr{x}{abc} - 1} \sum_{j=0}^{ab-1} ic+ \sum_{j=0}^{\fra{y}{ab}} \lfr{\fra{x}{c}+jc}{ab}+ \sum_{j=0}^{\fra{y}{ab}} \lfr{x}{abc}c \\=  \lfr{x}{abc}\sum_{j=0}^{ab-1} \lfr{\fra{x}{c}+jc}{ab}  + \frac{abc}{2}\left(\lfr{x}{abc}-1\right)\lfr{x}{abc} + \sum_{j=0}^{\fra{y}{ab}} \lfr{\fra{x}{c}+jc}{ab}+ (\fra{y}{ab}+1) \lfr{x}{abc}c $

$S2= \displaystyle \sum_{l=0}^{\lfr{x}{c}} \epsilon_{\langle a, b\rangle}(\fra{\fra{x}{c}+cl}{ab})= \sum_{i=0}^{\lfr{x}{abc} - 1} \sum_{j=0}^{ab-1} \epsilon_{\langle a, b\rangle}(\fra{\fra{x}{c}+(abi+j)c}{ab})  +  \sum_{j=0}^{\fra{y}{ab}} \epsilon_{\langle a, b\rangle}(\fra{\fra{x}{c}+(ab\lfr{x}{abc}+j)c}{ab}) =  \sum_{i=0}^{\lfr{x}{abc} - 1} \sum_{j=0}^{ab-1} \epsilon_{\langle a, b\rangle}(\fra{\fra{x}{c}+jc}{ab})  +  \sum_{j=0}^{\fra{y}{ab}} \epsilon_{\langle a, b\rangle}(\fra{\fra{x}{c}+jc}{ab}) \\=  \lfr{x}{abc} \sum_{j=0}^{ab-1} \epsilon_{\langle a, b\rangle}(\fra{\fra{x}{c}+jc}{ab})  +  \sum_{j=0}^{\fra{y}{ab}} \epsilon_{\langle a, b\rangle}(\fra{\fra{x}{c}+jc}{ab}) $ \\

Using Hermite reciprocity, since $\gcd(ab,c)=1$:  \\ $\displaystyle \sum_{j=0}^{ab-1} \lfr{\fra{x}{c}+jc}{ab} = \sum_{j=0}^{c-1} \lfr{\fra{x}{c}+jab}{c} =  \sum_{j=0}^{c-1} \lfr{\fra{x}{c}+\fra{jab}{c} + c\lfr{jab}{c}   }{c} =  \sum_{j=0}^{c-1} \lfr{\fra{x}{c}+\fra{jab}{c} }{c} + \sum_{j=0}^{c-1}\lfr{jab}{c} =  \sum_{l=0}^{c-1} \lfr{\fra{x}{c}+l }{c} + \sum_{j=0}^{c-1}\lfr{jab}{c} =  \floor*{c\frac{\fra{x}{c}}{c}}+\frac{(ab-1)(c-1)}{2} = \fra{x}{c}+\frac{(ab-1)(c-1)}{2} $ \\

Hence we obtain $S1=$ \\
 
$\displaystyle \sum_{j=0}^{\fra{y}{ab}} \lfr{\fra{x}{c}+jc}{ab} + \lfr{x}{abc}\sum_{j=0}^{ab-1} \lfr{\fra{x}{c}+jc}{ab}  + \frac{abc}{2}\left(\lfr{x}{abc}-1\right)\lfr{x}{abc} + (\fra{y}{ab}+1) \lfr{x}{abc}c \\= \sum_{j=0}^{\fra{y}{ab}} \lfr{\fra{x}{c}+jc}{ab} + \lfr{x}{abc}\left(\fra{x}{c}+\frac{(ab-1)(c-1)}{2}\right)+ \frac{abc}{2}\left(\lfr{x}{abc}-1\right)\lfr{x}{abc} + \\ (\fra{y}{ab}+1) \lfr{x}{abc}c = \sum_{j=0}^{\fra{y}{ab}} \lfr{\fra{x}{c}+jc}{ab} +  \lfr{x}{abc}\left( \frac{abc}{2}\left(\lfr{x}{abc}-1\right) + \\ \fra{x}{c}+\frac{(ab-1)(c-1)}{2} + \\ (\fra{y}{ab}+1)c  \right) \\=   \sum_{j=0}^{\fra{y}{ab}} \lfr{\fra{x}{c}+jc}{ab} + S3$, where $S3$ is the sum 
\\  $\displaystyle S3= \lfr{x}{abc}\left( \frac{abc}{2}\left(\lfr{x}{abc}-1\right) + \fra{x}{c}+\frac{(ab-1)(c-1)}{2} +  (\fra{y}{ab}+1)c  \right) \\= \lfr{x}{abc}\left( \frac{abc}{2}\left(\lfr{x}{abc}-1\right) + \fra{x}{c}+\frac{(ab-1)(c-1)}{2} + \\ (\fra{\lfr{x}{c}}{ab}+1)c  \right) \\= \lfr{x}{abc}\left( \frac{abc}{2}\left(\lfr{x}{abc}-1\right) + x-c\lfr{x}{c}+\frac{(ab-1)(c-1)}{2} +  (\lfr{x}{c}-ab\lfr{x}{abc}+1)c  \right) \\= \lfr{x}{abc}\left( \frac{abc}{2}\lfr{x}{abc}-\frac{abc}{2} + x-c\lfr{x}{c}+\frac{abc-ab-c+1}{2} + c\lfr{x}{c}-abc\lfr{x}{abc}+c  \right) \\= \lfr{x}{abc}\left(x-\frac{abc}{2}\lfr{x}{abc} +\frac{c-ab+1}{2} \right) = \lfr{x}{abc}\left(\frac{abc}{2}\lfr{x}{abc} + \fra{x}{abc} + \frac{c-ab+1}{2} \right)$

Finally we get $\displaystyle S1=  \lfr{x}{abc}\left(\frac{abc}{2}\lfr{x}{abc} + \fra{x}{abc} + \frac{c-ab+1}{2} \right) +  \sum_{j=0}^{\fra{\lfr{x}{c}}{ab}} \lfr{\fra{x}{c}+jc}{ab} \\=  \frac{abc}{2}\lfr{x}{abc}^2 +\lfr{x}{abc}\left( \fra{x}{abc} + \frac{c-ab+1}{2} \right) +\sum_{j=0}^{\fra{\lfr{x}{c}}{ab}} \lfr{\fra{x}{c}+jc}{ab}$\\
Notice that the main term is $\displaystyle  \frac{abc}{2}\lfr{x}{abc}^2$ which has order $\displaystyle \sim \frac{x^2}{2abc}$.\\

The entire result is then $\displaystyle S1+S2 \\=  \frac{abc}{2}\lfr{x}{abc}^2 +\lfr{x}{abc}\left( \fra{x}{abc} + \frac{c-ab+1}{2} \right) +\sum_{j=0}^{\fra{\lfr{x}{c}}{ab}} \lfr{\fra{x}{c}+jc}{ab} + \\ \lfr{x}{abc} \sum_{j=0}^{ab-1} \epsilon_{\langle a, b\rangle}(\fra{\fra{x}{c}+jc}{ab})  +  \sum_{j=0}^{\fra{y}{ab}} \epsilon_{\langle a, b\rangle}(\fra{\fra{x}{c}+jc}{ab})$. \\ 

Each of the sums have easy to find bounds: 
\begin{itemize}

\item $\displaystyle \fra{x}{abc} + \frac{c-ab+1}{2} < abc+\frac{c-ab+1}{2}$. 

\item $\displaystyle \sum_{j=0}^{\fra{\lfr{x}{c}}{ab}} \lfr{\fra{x}{c}+jc}{ab} \leq \displaystyle \sum_{j=0}^{ab-1} \lfr{\fra{x}{c}+jc}{ab} = \fra{x}{c}+\frac{(ab-1)(c-1)}{2} \leq c+\frac{(ab-1)(c-1)}{2}$.

\item $\displaystyle \sum_{j=0}^{ab-1} \epsilon_{\langle a, b\rangle}(\fra{\fra{x}{c}+jc}{ab}) \leq  \sum_{j=0}^{ab-1} 1= ab$

\item $\displaystyle \sum_{j=0}^{\fra{y}{ab}} \epsilon_{\langle a, b\rangle}(\fra{\fra{x}{c}+jc}{ab}) \leq  \sum_{j=0}^{ab-1} \epsilon_{\langle a, b\rangle}(\fra{\fra{x}{c}+jc}{ab}) \leq ab$

\end{itemize}

Hence we obtain:

\begin{thm}\label{thm-13} $n=3$, $M = \langle a, b, c \rangle$, $\gcd(a,b)=\gcd(a,c)=\gcd(b,c)=1$. We have $$T_{M}(x) = \frac{abc}{2}\lfr{x}{abc}^2 + r_{M}(x)$$ where $|r_M(x)|\leq Kx+L$ for absolute constants $K, L$ (depending only on $a, b, c$).

\end{thm} 

\section{Dimension $n$ arbitrary}

The technique above generalizes, by summing up the leading terms using Corollary \ref{cor-4} obtaining:

\begin{thm} \label{thm-14} For $M=\langle a_1, ... , a_d\rangle$, $\gcd(a_i,a_j)=1$ for all $i\neq j$. The counting function $T_{M}(x)$ has main order $$T_M(x) \sim \frac{(a_1....a_d)^{d-2}}{(d-1)!}\lfr{x}{a_1.... a_d}^{d-1}$$ meaning that $T_{M}(x)=\frac{(a_1...a_d)^{d-2}}{(d-1)!}\lfr{x}{a_1.... a_d}^{d-1} + r_{M}(x)$, where $|r_{M}(x)|$ is bounded by a polynomial in $x$ (depending only on $M$) of degree at most $d-2$.
\end{thm}

\Addresses


\begin{thebibliography}{1}

\bibitem{BLDPS} Beck, M.; De Loera, J. A.; Develin, M.; Pfeifle, J.; Stanley, R. P. (2005), ``Coefficients and roots of Ehrhart polynomials", Integer points in polyhedra- geometry, number theory, algebra, optimization, Contemp. Math., 374, Providence, RI: Amer. Math. Soc., pp. 15-36

\bibitem{RS} J. C. Rosales, J. C and P. A. Garc\'{i}a-S\'{a}nchez, P. A. Numerical semigroups. Developments in Mathematics, 20. Springer, New York, 2009


\end{thebibliography}
\end{document}